\documentclass[12pt, reqno, twoside, a4paper]{article}

\usepackage{amsmath,amssymb,amsthm,amsfonts,mathtools}
\usepackage{mathrsfs}
\usepackage{enumerate}
\usepackage{natbib}
\usepackage{multirow}
\usepackage[colorlinks=true,linkcolor=blue,citecolor=blue,urlcolor=blue]{hyperref}
\usepackage{authblk}
\usepackage{orcidlink}
\usepackage{todonotes}
\usepackage[margin=1in]{geometry}
\usepackage{fancyhdr}
\newcommand\blfootnote[1]{%
  \begingroup
  \renewcommand\thefootnote{}\footnote{\hspace{-1.8em}#1}%
  \addtocounter{footnote}{-1}%
  \endgroup
}

\newtheorem{theorem}{Theorem}[section]
\newtheorem{lemma}[theorem]{Lemma}
\newtheorem{proposition}[theorem]{Proposition}
\newtheorem{corollary}[theorem]{Corollary}

\theoremstyle{definition}
\newtheorem{definition}[theorem]{Definition}

\newtheorem*{acknow}{\textup{Acknowledgement}}

\newcommand{\Cp}{C_p(X, G)}
\newcommand{\N}{\mathbb{N}}
\newcommand{\R}{\mathbb{R}}
\newcommand{\U}{\mathcal{U}}
\newcommand{\V}{\mathcal{V}}
\newcommand{\A}{\mathcal{A}}
\newcommand{\B}{\mathcal{B}}
\newcommand{\sfin}{S_{fin}(\Omega(X), \Omega(X))}
\newcommand{\sone}{S_1(\Omega(X), \Omega(X))}
\newcommand{\gfin}{G_{fin}(\Omega(X), \Omega(X))}
\newcommand{\gone}{G_1(\Omega(X), \Omega(X))}

\newcommand{\gfinY}{G_{fin}(\Omega(Y), \Omega(Y))}
\newcommand{\goneY}{G_1(\Omega(Y), \Omega(Y))}

\title{Countable Fan Tightness and Selection Games in 
Group-Valued Function Spaces}
\author{Souvik Mandal$^{\,a, *}$ \orcidlink{0009-0004-0429-3063}}
\author{Ankur Sarkar$^{\,b}$ \orcidlink{0009-0002-3320-2568}}

\affil{\footnotesize\itshape $^{a}$Department of Mathematics, Indian Institute of Technology Madras, \\ 
\footnotesize\itshape Chennai 600036, Tamil Nadu, India.}
\affil{\footnotesize\itshape $^{b}$The Institute of Mathematical Sciences, A CI of Homi Bhabha National Institute, \\ 
\footnotesize\itshape Chennai 600113, Tamil Nadu, India.}
\date{}
\begin{document}
\maketitle
\blfootnote{$^{*}$ \textit{Corresponding author.}}
\blfootnote{\raggedright\hspace{0.50em}
  \makebox[8.5em][l]{\textit{E-mail addresses:}}%
  \begin{minipage}[t]{0.75\textwidth}
    \texttt{ma22d014@smail.iitm.ac.in}, \texttt{ssouvik.xyz@gmail.com} (S. Mandal); \\
    \texttt{ankurimsc@gmail.com} (A. Sarkar).
  \end{minipage}}
  \vspace{-3.5em}

\author{}


\begin{abstract}
\noindent Game-theoretic characterizations of selection principles 
provide a powerful framework for analyzing covering 
properties through strategic interactions. 
For a Tychonoff space $X$ and a non-trivial metrizable 
arc-connected topological group $G$, we prove that 
Player~II has a winning strategy in the $\Omega$-Menger 
game on $X$ if and only if Player~II has a winning 
strategy in the countable fan tightness game on 
$C_p(X, G)$ at the identity function. 
The analogous equivalence is established between the 
$\Omega$-Rothberger game on $X$ and the countable 
strong fan tightness game on $C_p(X, G)$ at the 
identity function. 
These results extend the game-theoretic 
characterizations of Clontz from $G = \mathbb{R}$ to 
arbitrary metrizable arc-connected groups, and lift the 
selection-principle equivalences of Ko\v{c}inac to the 
game-theoretic setting. As consequences, we establish 
that the game-theoretic tightness properties of 
$C_p(X,G)$ are independent of $G$, preserved under 
$G$-equivalence, and remain valid for Markov 
strategies.
\end{abstract}
\vspace{-0.5em}
\begin{center}
\begin{minipage}{0.845\textwidth}
    \footnotesize
    \begin{list}{}{%
        \leftmargin=4.3em 
        \labelwidth=5em
        \labelsep=0pt \parsep=0pt \topsep=0pt \itemsep=0pt
    }
        \item[\textit{Keywords:}\hfill] Selection principles, 
Topological games, Fan tightness, $\omega$-cover, 
Topological groups, $C_p$-theory.
    \end{list}

    \vspace{2pt} 

    \begin{list}{}{%
        \leftmargin=17.8em 
        \labelwidth=18.5em
        \labelsep=0pt \parsep=0pt \topsep=0pt \itemsep=0pt
    }
        \item[2020\hspace{1mm}\textit{Mathematics Subject 
Classification:}\hfill] Primary 54C35, 54D20, 91A44;\\ 
Secondary 54H11, 22A05.
    \end{list}
\end{minipage}
\end{center}

\section{Introduction}
The theory of selection principles bridges general topology, infinite combinatorics, and topological game theory.
Originally formalized to characterize special sets of real numbers in measure theory, selection principles classify topological spaces based on their interaction with sequences of open covers; we refer the reader to \cite{Scheepers3, Scheepers, jmss, Arh} for a detailed treatment.

Concurrently, in the theory of $C_p$-spaces, following work of Arhangel'skii \cite{arhq,Arh}, covering properties of a Tychonoff space $X$ are characterized by convergence properties of $C_p(X,\mathbb{R})$, the space of continuous real-valued functions endowed with the topology of pointwise convergence. 
A classical theorem, established by Arhangel'skii~\cite[Theorem~4]{arhq} and reformulated in terms of $\omega$-covers by Just, Miller, Scheepers, and Szeptycki~\cite[Theorem~3.9]{jmss}, states that $\Omega$-Menger property holds if and only if $C_p(X, \R)$ has countable fan tightness; an alternative proof can be found in~\cite[Theorem~22]{Clontz}.

In the game-theoretic context, Scheepers~\cite{Scheepers3} proved that Player~II has a winning strategy in the $\Omega$-Menger game on $X$ if and only if Player~II has a winning strategy in the fan tightness game on $C_p(X,\mathbb{R})$.
Clontz~\cite{Clontz} subsequently established these 
equivalences at the game-theoretic level, but remains open for arbitrary metrizable groups. 

More recently, \( C_p \)-theory has been extended by replacing \( \mathbb{R} \) with an arbitrary topological group \( G \). 
Shakhmatov and Sp\v{e}v\'ak~\cite{SS} gave a systematic study of \( C_p(X,G) \) and established its basic topological properties. Ko\v{c}inac~\cite{Koc2} established a selection-principle equivalence staing that $\Omega$-Menger property on $X$ is 
equivalent to the countable fan tightness on $C_p(X,G)$, provided that $G$ is a metrizable topological group and $X$ is $G^*$-regular.
Mishra, Pandey, Ravindran, and Yadav~\cite{Mispro} further developed these selection-principle equivalences. 
However, these works do not address topological games and remain entirely within the framework of selection principles.

Theorems~\ref{thm:main_game} and~\ref{cor:rothberger} 
establish the equivalence between the $\Omega$-Menger 
game on $X$ and the countable fan tightness game on 
$C_p(X,G)$, and between the $\Omega$-Rothberger game 
on $X$ and the countable strong fan tightness game on 
$C_p(X,G)$, respectively, extending the results of 
Scheepers~\cite{Scheepers3} and 
Clontz~\cite{Clontz} from $G = \R$ to arbitrary 
metrizable arc-connected groups, and lifting the 
selection-principle equivalences of 
Ko\v{c}inac~\cite{Koc2} to the game-theoretic setting.
The known results and the position of the present 
contribution are summarized in the following table.
\begin{center}
\renewcommand{\arraystretch}{1.3}
\begin{tabular}{|c|c|p{5.4cm}|p{5cm}|}
\hline
& & \textbf{Selection principle} & \textbf{Game-theoretic}\\
\hline
\multirow{2}{*}{$G = \R$} 
& $\Omega$-Menger 
& Arhangel'skii~\cite[Theorem~4]{arhq}+ 
  Just et al.~\cite[Theorem~3.9]{jmss} 
& Scheepers~\cite{Scheepers3}, Clontz~\cite[Lemma~20]{Clontz}\\
\cline{2-4}
& $\Omega$-Rothberger 
& Sakai~\cite[Theorem~1]{propc}+ 
  Just et al.~\cite[Theorem~3.8]{jmss} 
& Scheepers~\cite{Scheepers3}, Clontz~\cite[Theorem~43]{Clontz}\\
\hline
\multirow{2}{*}{Metrizable $G$} 
& $\Omega$-Menger 
& Ko\v{c}inac~\cite[Corollary~2.4]{Koc2} 
& \textbf{Theorem~\ref{thm:main_game}}\\
\cline{2-4}
& $\Omega$-Rothberger 
& Ko\v{c}inac~\cite[Theorem~2.5]{Koc2} 
& \textbf{Theorem~\ref{cor:rothberger}}\\
\hline
\end{tabular}
\end{center}
The extension from $\mathbb{R}$ to a general metrizable group $G$ requires arc-connectedness of $G$, guaranteeing continuous paths from the identity to arbitrary elements, while for the space $X$, the Tychonoff property suffices and no normality assumption is needed.

\section{Preliminaries and Definitions}

Throughout this paper, $X$ denotes an infinite Tychonoff 
(completely regular Hausdorff) space, $\mathcal{O}(X)$ denotes the collection of open covers of $X$ and $G$ a non-trivial 
metrizable arc-connected topological group with identity 
element $e$. 
By the Birkhoff--Kakutani theorem 
\cite[Theorem~3.3.12]{birkaku}, a (Hausdorff) topological group is first-countable if and only if it is 
metrizable.

Since \(G\) is arc-connected, for any \(g \in G \setminus \{e\}\) there exists an embedding \(\gamma\colon [0,1] \hookrightarrow G\) with \(\gamma(0)=e\) and \(\gamma(1)=g\), and we refer to such a map \(\gamma\) as an arc in \(G\). 
We note from \cite[Corollary~31.6]{Willard} that, in our setting, arc-connectedness is equivalent to path-connectedness.

The space of all continuous functions from $X$ to $G$, equipped with the topology of pointwise convergence, is denoted by $\Cp$. 
The constant function $\mathbf{e} : X \to G$, defined by $\mathbf{e}(x) = e$ for all $x \in X$, will be used throughout. 
The family 
\[
\{W(x,U)\colon x\in X, \hspace{1mm}U \text{ is an open subset of } G \},
\] where 
\[
W(x,U)=\{f\in C_{p}(X,G)\,\colon\,f(x)\in U\,\}
\] forms a subbasis for the topology of $C_{p}(X,G)$. 
For any $f\in \Cp$, finite subset $F$ of $X$ and $\epsilon>0$, the set 
\[W_f(F, \varepsilon) := \{ h \in \Cp : d(h(x), f(x)) < 
\varepsilon \text{ for all } x \in F \}\] is a basic open neighbourhood of $f$, where $d$ is a metric on $G$. 
Note that 
\[W_f(F,\epsilon)=\bigcap_{x\in F} W\bigl(x,\, B_d(f(x).
\varepsilon)\bigr).\] 

We recall the general framework of selection principles and their associated infinite games, following Scheepers~\cite{Scheepers3,chp8}.
Let $\A$ and $\B$ be collections of subsets of an infinite set $S$.
\begin{definition}{\cite{Scheepers3}}
The \textit{selection principle} $S_{fin}(\A, \B)$ states: for every sequence $(A_n)_{n \in \N}$ of elements of $\A$, there exist finite subsets $B_n \subset A_n$ such that $\bigcup_{n \in \N} B_n \in \B$.
\end{definition}

\begin{definition}{\cite{Scheepers3}}
The \textit{selection principle} $S_1(\A, \B)$ states: for every sequence $(A_n)_{n \in \N}$ of elements of $\A$, there exist elements $b_n \in A_n$ such that $\{b_n : n \in \N\} \in \B$.
\end{definition}

\begin{definition}{\cite{Scheepers3}}\label{def:gfin_general}
The \textit{selection game} $G_{fin}(\A, \B)$ is an infinite game played over $\omega$ innings between two players, I and II.
In inning $n$, Player~I selects $A_n \in \A$; Player~II responds with a finite subset $B_n \subset A_n$.
Player~II wins if $\bigcup_{n \in \N} B_n \in \B$; otherwise Player~I wins.
\end{definition}

\begin{definition}{\cite{Scheepers3}}
The \textit{selection game} $G_1(\A, \B)$ is defined analogously to $G_{fin}(\A, \B)$, except that in each inning $n$, Player~II selects a \emph{single} element $b_n \in A_n$.
Player~II wins if $\{b_n : n \in \N\} \in \B$; otherwise Player~I wins.
\end{definition}

\begin{definition}{\cite{Scheepers3}}
A \textit{strategy} for Player~II in $G_{fin}(\A, \B)$ is 
a function
\[
\sigma \colon \bigcup_{n=1}^{\infty} \A^n \to 
\bigcup_{A \in \A} [A]^{<\omega}
\]
satisfying $\sigma(A_1, \ldots, A_n) \subset A_n$ for each 
$(A_1, \ldots, A_n) \in \A^n$; here $[A]^{<\omega}$ denotes 
the collection of all finite subsets of $A$.
That is, in inning $n$, having seen Player~I's moves 
$A_1, \ldots, A_n$, Player~II responds with the finite 
subset $B_n = \sigma(A_1, \ldots, A_n) \subset A_n$.
A strategy $\sigma$ is \textit{winning} if whenever 
Player~I plays $A_n \in \A$ in inning $n$, Player~II 
wins the game by responding with 
$\sigma(A_1, \ldots, A_n)$ in each inning $n$.
We write $\textup{II} \uparrow G_{fin}(\A, \B)$ to indicate 
that Player~II possesses a winning strategy in 
$G_{fin}(\A, \B)$.

A strategy for Player~II in $G_1(\A, \B)$ is defined analogously, with $\sigma(A_1, \ldots, A_n) \in A_n$. 
\end{definition}
\begin{definition}\cite{chp8}
A strategy $\sigma$ is called \textit{Markov} if there 
exists a function 
\[
\tau \colon \A \times \N \to 
\bigcup_{A \in \A} [A]^{<\omega}
\]with 
$\tau(A, n) \subset A$ for all $A \in \A$ and 
$n \in \N$, such that $\sigma(A_1, \ldots, A_n) = \tau(A_n, n)$. That is, Player~II's response depends only on 
Player~I's current move and the round number.
\end{definition}
We may observe that if Player~II has a winning strategy in 
$G_{fin}(\A, \B)$, then $S_{fin}(\A, \B)$ holds, but the 
converse need not follow.
\begin{definition}{\cite{Scheepers3}}\label{def:omega_cover}
An open cover $\U$ of $X$ is an \textit{$\omega$-cover} if $X \notin \U$ and every finite subset of $X$ is contained in some member of $\U$. 
We denote by $\Omega(X)$ the collection of all $\omega$-covers of $X$.
\end{definition}
In this paper, we focus on instances of this framework involving $\omega$-covers of $X$. 
Choosing $\A = \B = \mathcal{O}(X)$ recovers the classical Menger and Rothberger properties and their associated games, while $\A = \B = \Omega(X)$ yields the $\Omega$-Menger property $\sfin$, the $\Omega$-Rothberger property $\sone$, the $\Omega$-Menger game $\gfin$, and the $\Omega$-Rothberger game $\gone$, which are the focus of this paper.

The tightness properties of function spaces arise as instances of the same general framework. 
For a topological space $Y$ and $y \in Y$, let $\Omega_y = \{A \subset Y : y \in \overline{A} \setminus A\}$, the family of sets with $y$ as a limit point~\cite{Scheepers3}. 
\begin{definition}{\cite{Scheepers3}}
A space $Y$ has \textit{countable fan tightness} at $y \in Y$ if $S_{fin}(\Omega_y, \Omega_y)$ holds; that is, for every sequence $(A_n)_{n \in \N}$ with $y \in \bigcap_{n \in \N}\overline{A_n} \setminus A_n$, there exist finite $B_n \subset A_n$ with $y \in \overline{\bigcup_{n \in \N} B_n}$. A space $Y$ has \textit{countable fan tightness} if it 
has countable fan tightness at every point.
\end{definition}
 
\begin{definition}{\cite{Scheepers3}}
A space $Y$ has \textit{countable strong fan tightness} at $y \in Y$ if $S_1(\Omega_y, \Omega_y)$ holds; that is, for every sequence $(A_n)_{n \in \N}$ with $y \in\bigcap_{n \in \N} \overline{A_n} \setminus A_n$, there exist elements $b_n \in A_n$ with $y \in \overline{\{b_n : n \in \N\}}$. A space $Y$ has \textit{countable strong fan tightness} if it 
has countable strong fan tightness at every point.
\end{definition}
 The corresponding games are:
\begin{itemize}
\item The \textit{countable fan tightness game} $CFT(Y, y) := G_{fin}(\Omega_y, \Omega_y)$: in inning $n$, Player~I chooses $A_n \subset Y \setminus \{y\}$ with $y \in \overline{A_n}$; Player~II selects finite $B_n \subset A_n$.
Player~II wins if $y \in \overline{\bigcup_{n \in \N} B_n}$.
\item The \textit{countable strong fan tightness game} $SCFT(Y, y) := G_1(\Omega_y, \Omega_y)$: in inning $n$, Player~I chooses $A_n \subset Y \setminus \{y\}$ with $y \in \overline{A_n}$; Player~II selects a single $b_n \in A_n$.
Player~II wins if $y \in \overline{\{b_n : n \in \N\}}$.
\end{itemize}
Finally, we recall a notion of equivalence for spaces 
induced by their group-valued function spaces.
\begin{definition}[\cite{SS,osi4}]
Two spaces $X$ and $Y$ are said to be \textit{$G$-equivalent}, denoted $X \sim_G Y$, if $C_p(X,G)$ and $C_p(Y,G)$ are isomorphic as topological groups.

In the case $G=\mathbb{R}$, this coincides with the classical notion of \textit{$l$-equivalence}; if, instead, $C_p(X,\mathbb{R})$ and $C_p(Y,\mathbb{R})$ are merely homeomorphic as topological spaces, then $X$ and $Y$ are said to be \textit{$t$-equivalent}.
\end{definition}

\section{The Characterization Theorems}
The main results of this article are the following two 
theorems.
\begin{theorem}\label{thm:main_game}
Let $X$ be an infinite Tychonoff space and $G$ a non-trivial metrizable arc-connected topological group.
Then the following are equivalent:
\begin{enumerate}
\item[(a)] Player~II has a winning strategy in the $\Omega$-Menger game $\gfin$. 
\item[(b)] Player~II has a winning strategy in the countable fan tightness game $CFT(\Cp, \mathbf{e})$.
\end{enumerate}
\end{theorem}
\begin{theorem}\label{cor:rothberger}
Let $X$ be an infinite Tychonoff space and $G$ a non-trivial metrizable arc-connected topological group.
The following are equivalent:
\begin{enumerate}
    \item[(a)] Player~II has a winning strategy in the $\Omega$-Rothberger  game $\gone$. 
    \item[(b)] Player~II has a winning strategy in the countable strong fan tightness game $SCFT(\Cp, \mathbf{e})$.
\end{enumerate}
\end{theorem}
Before proving the above two theorems, we establish the following lemmas, which will be used in the proofs.
\begin{lemma}\label{lem:closure}
Let \( A \subset \Cp \). Then \( \mathbf{e} \in \overline{A} \) if and only if for every finite subset \( F \subset X \) and every \( \epsilon > 0 \), there exists \( f \in A \) such that \( d(f(x), e) < \epsilon \) for all \( x \in F \), where $d$ denotes the metric on $G$.
\end{lemma}
\begin{proof}
Since the collection 
$\{ W_{\mathbf{e}}(F, \varepsilon) : F \subset X 
\text{ finite},\; \varepsilon > 0 \}$ forms a 
neighborhood base for $\mathbf{e}$, we have 
$\mathbf{e} \in \overline{A}$ if and only if 
$W_{\mathbf{e}}(F, \varepsilon) \cap A \neq \emptyset$ 
for every finite $F \subset X$ and every 
$\varepsilon > 0$, which is precisely the stated 
condition.
\end{proof}
\begin{lemma}\label{lem:test_function}
Fix $g \in G \setminus \{e\}$ and an arc $\gamma : [0,1] \hookrightarrow G$ with $\gamma(0) = e$ and $\gamma(1) = g$. Let $U \subsetneq X$ be an open set and $F \subset U$ a finite subset. Then there exists a map $f_{F,U} \in C_p(X, G)$ such that
\begin{enumerate}
    \item[\textup{(i)}] $f_{F,U}(x) = e$ for all $x \in F$,
    \item[\textup{(ii)}] $f_{F,U}(x) = g$ for all $x \in X \setminus U$.
\end{enumerate}
\end{lemma}
\begin{proof}
Suppose $F=\{x_1,\ldots,x_k\}$. By complete regularity of $X$, for each $x_i \in F$ there exists a continuous map $g_i\colon X \to [0,1]$ such that $g_i(x_i)=0$ and $g_i|_{X\setminus U}=1$. We then define $\varphi\colon X \to [0,1]$ by $\varphi(y)=\min_{1\le i\le k} g_i(y)$, which is continuous on $X$. Now set $f_{F,U}=\gamma\circ \varphi$. Then $f_{F,U}|_F=\gamma(0)=e$ and $f_{F,U}|_{X\setminus U}=\gamma(1)=g$.
\end{proof}
We now prove the two theorems.
\begin{proof}[Proof of Theorem~\ref{thm:main_game}]
Let $\sigma$ be a winning strategy for Player~II in $G_{fin}(\Omega(X), \Omega(X))$. We construct a winning strategy for Player~II in $CFT(\Cp, \mathbf{e})$.

Since $X$ is infinite and $T_1$, every finite subset $F$ of $X$ satisfies $F\subset X\setminus \{x\}$ for some $x\in X\setminus F$. Hence $\U_0 = \{ X \setminus \{x\} : x \in X \}$ is an $\omega$-cover of $X$.

Suppose that Player~I selects subsets \( A_n \subset \Cp \) such that \( \mathbf{e} \in \overline{A_n} \setminus A_n \) for every \( n \). For each \( n \) and each \( f \in A_n \), define
\begin{equation*}
U_{f,n} := \{ x \in X : d(f(x), e) < 1/n \}.
\end{equation*}
Let
\[
\U_n' := \{ U_{f,n} : f \in A_n,\; U_{f,n} \neq X \}.
\]
We say that the \( n \)-th inning is non-degenerate if \( \U_n' \) forms an \( \omega \)-cover of \( X \); otherwise, it is called degenerate.

Suppose that \( \U_n' \) is not an \( \omega \)-cover of \( X \). Then, by Definition~\ref{def:omega_cover}, there exists a finite set \( K \subset X \) such that \( K \) is not contained in any member of \( \U_n' \). Since \( \mathbf{e} \in \bigcap_{n\in \mathbb{N}} \overline{A_n} \), Lemma~\ref{lem:closure} yields that there exists \( f_n \in A_n \) such that  $d(f_n(x), e) < \frac{1}{n}$ for all  $x \in K$. 
Hence \( K \subset U_{f_n,n} \). It follows from the choice of \( K \) that \( U_{f_n,n} \notin \U_n' \), and therefore \( U_{f_n,n} = X \). Consequently, if $n$-th inning is degenerate, then there exists \( f_n \in A_n \) such that \( U_{f_n,n} = X \).

Set
\[
\widetilde{\U}_n =
\begin{cases}
\U_n' & \text{if inning $n$ is non-degenerate,}\\
\U_0 & \text{if inning $n$ is degenerate.} 
\end{cases}
\]
Note that \( \widetilde{\U}_n \) is an \( \omega \)-cover of \( X \). Since \( \sigma \) is a winning strategy for Player~II in the game \( G_{fin}(\Omega(X), \Omega(X)) \), it follows from Definition~\ref{def:gfin_general} that for each \( n \) there exists a finite set \( \V_n \subset \widetilde{\U}_n \) such that $\mathcal{C} := \bigcup_{n \in \mathbb{N}} \V_n$ is an \( \omega \)-cover of \( X \).

We define Player~II’s response in the game \( CFT(\Cp, \mathbf{e}) \) by
\[
B_n =
\begin{cases}
\{ f \in A_n : U_{f,n} \in \V_n \}, & \text{if the \( n \)-th inning is non-degenerate},\\
\{ f_n \}, & \text{if the \( n \)-th inning is degenerate},
\end{cases}
\]
where, in the degenerate case, \( f_n \in A_n \) is as obtained above with \( U_{f_n,n} = X \).  
Note that in either case, \( B_n \) is a finite subset of \( A_n \). Now to show that the choice of $B_i$'s gives winning strategy for Player II in the countable fan tightness game $CFT(C_p(X,G)),\mathbf{e})$, it suffices to show that $\mathbf{e} \in \overline{\bigcup_{n} B_n}$.

Let $\delta>0$ and $F$ be a finite subset of $X$. Since $W_{\mathbf{e}}(F,\delta)=\{f\in C_p(X,G): 
d(f(x),e)< \delta \text{ for all } x\in F\}$ is a basic 
open neighbourhood of $\mathbf{e}$, we choose $m\in \mathbb{N}$ such that $\frac{1}{m}<\delta$. The we consider the collection $\mathcal{C}_{< m}: \bigcup\limits_{k=1}^{m-1}\V_k$. Clearly $\mathcal{C}_{< m}$ is finite and every member of $\mathcal{C}_{< m}$ is a proper subset of $X$. Now for each $U\in \mathcal{C}_{< m}$, we pick a point $p_U\in X\setminus U$. Then 
\[F^* = F \cup \{ p_U : U \in \mathcal{C}_{<m} \}\]
is a finite subset of $X$ and is not contained in any member of $\mathcal{C}_{< m}$. Since $\mathcal{C}$ is an $\omega$-cover, there exists $k \ge m$ and $V \in \V_k$ with $F^* \subset V$.

If inning $k$ is non-degenerate, then $V=U_{f,k}$ for some $f\in B_k$. Hence 
\[
d(f(x), e) < \frac{1}{k} \le \frac{1}{m} < \delta
\quad \text{for all } x \in F.
\]
If inning $k$ is degenerate, then there exists $f_k\in B_k$ such that $d(f_k(x),e)< \frac{1}{k}$ for all $x\in X$. Thus 
\[d(f_k(x),e)< \frac{1}{k}\leq \frac{1}{m}< \delta ~ \text{ for all } x\in F.\]
Hence, in both cases, \( W_{\mathbf{e}}(F,\delta) \cap \bigcup_{n} B_n \neq \emptyset \). Consequently, \( \mathbf{e} \in \overline{\bigcup_{n} B_n} \).

Conversely, suppose that \( \tau \) is a winning strategy for Player~II in the countable fan tightness game \( CFT(\Cp, \mathbf{e}) \). Our aim is to construct a winning strategy for Player II in the game $G_{fin}(\Omega(X),\Omega(X))$.

Let Player~I choose \( \U_n \in \Omega(X) \) in inning $n$. Since \( G \) is a non-trivial topological group, there exists \( g \in G \) with \( g \neq e \). Fix such an element \( g \). As \( G \) is arc connected, it follows from Lemma~\ref{lem:test_function} that for every \( U \in \U_n \) and every finite set \( F \subset U \), there exists a function \( f_{F,U} \in C_p(X,G) \) such that  
\begin{equation*}
f_{F,U}\vert_F = e \quad \text{and} \quad f_{F,U}\vert_{X \setminus U} = g.
\end{equation*}
Now, we define $A_n = \bigl\{\, f_{F,U} : \; U \in \U_n,\; F \subset U \text{ finite}\,\bigr\}$.
Clearly $\mathbf{e}\not\in A_n$ for each $n$. We show that $\mathbf{e}\in \overline{A_n}$ for any $n$. 

Let \( K \subset X \) be a finite set and let \( \epsilon > 0 \). Since \( \U_n \) is an \( \omega \)-cover, Definition~\ref{def:omega_cover} implies that there exists \( U \in \U_n \) such that \( K \subset U \). For this choice of \( U \), the function \( f_{K,U} \in A_n \) satisfies \( f_{K,U}(x) = e \) for all \( x \in K \), and hence \( f_{K,U} \in W_{\mathbf{e}}(K,\epsilon) \). It follows from Lemma~\ref{lem:closure} that \( \mathbf{e} \in \overline{A_n} \).

Since \( \tau \) is a winning strategy for Player~II in the game \( CFT(C_p(X,G), \mathbf{e}) \), at the \( n \)-th inning, after Player~I chooses \( A_n \), Player~II selects a finite set \( B_n \subset A_n \) according to \( \tau \) so that $\mathbf{e} \in \overline{\bigcup_{n \in \mathbb{N}} B_n}$.
Note that for each \( h \in B_n \), there exist \( U_h \in \U_n \) and a finite set \( F_h \subset U_h \) such that \( h = f_{F_h, U_h} \). Define $\V_n := \{ U_h : h \in B_n \}$. To verify that this choice of \( \V_n \) yields a winning strategy for Player~II in the game \( G_{fin}(\Omega(X), \Omega(X)) \), it remains to show that $\bigcup_{n \in \mathbb{N}} \V_n$ is an \( \omega \)-cover of \( X \).

Let $F \subset X$ be an arbitrary finite set. Consider the basic open neighborhood $W_{\mathbf{e}}(F, r)$ of $\mathbf{e}$, where $r = d(g, e) > 0$. Since $\mathbf{e} \in \overline{\bigcup\limits_{n\in \mathbb{N}} B_n}$, there exists $h \in \bigcup\limits_{n\in \mathbb{N}} B_n$ with $h \in W_{\mathbf{e}}(F, r)$. Hence $d(h(x), e) < r$ for all $x \in F$. As $h\in \bigcup\limits_{n\in \mathbb{N}} B_n$, $h=f_{F_h, U_h}$ for some $U_h\in \U_k$ and finite subset $F_h$ of $U_h$. We claim that $F\subset U_h$. If not, then there exists a $x\in F$ such that $x\not\in U_h$. Therefore $h(x)= f_{F_h,U_h}(x)=g$ and so $d(h(x),e)=r$, contradicting that $h\in W_{\mathbf{e}}(F,r)$. Hence $F\subset U_h$ and $U_h\in \V_k\subset \bigcup\limits_{n\in \mathbb{N}} \V_n$. Since $\U_k$ is an $\omega$-cover, $U_h \neq X$. Consequently, $\bigcup\limits_{n\in \mathbb{N}} \V_n$ is an $\omega$-cover of $X$. This completes the proof. 
\end{proof}
The proof of Theorem~\ref{thm:main_game} extends to the $\Omega$-Rothberger analogue by replacing finite selections with single-element selections.
\begin{proof}[Proof of Theorem~\ref{cor:rothberger}]
The proof adapts the argument of 
Theorem~\ref{thm:main_game} to the $G_1$ setting.

Suppose Player~II has a winning strategy $\sigma$ in $\gone$. 
Given $A_n \subset \Cp$ with $\mathbf{e} \in \overline{A_n} \setminus A_n$, construct the $\omega$-covers $\widetilde{\U}_n$ exactly as in 
Theorem~\ref{thm:main_game}. 
Apply $\sigma$ to select $V_n \in \widetilde{\U}_n$ at each inning. Since $\sigma$ is winning, $\{V_n : n \in \N\}$ is an $\omega$-cover of $X$. 
If inning $n$ is non-degenerate, set $b_n = f_n$ where $V_n = U_{f_n, n}$; if degenerate, set $b_n = f_n$ as before. 
We now show that $\mathbf{e} \in \overline{\{b_n : n \in \N\}}$. For any finite subset $F \subset X$ and any $\delta > 0$, 
the set $W_{\mathbf{e}}(F,\delta)$ is a basic open 
neighbourhood of $\mathbf{e}$. Choose $m \in \N$ with 
$\frac{1}{m} < \delta$. Since each $V_i$ is a finite subset of $X$, we choose a point $p_{V_i}\in X\setminus V$. 
Then $F^* = F \cup \{p_{V_1}, \ldots, p_{V_{m-1}}\}$ is a finite subset of $X$ that is not contained in any $V_i$ for $i\leq m-1$. 
As $\{V_n: n\in \mathbb{N}\}$ is an $\omega$-cover, there exists $k \geq m$ such that $F^* \subset V_k$. 
Then, for both degenerate and non-degenerate $k$, we have
\[
d(b_k(x), e) < \frac{1}{k} \leq \frac{1}{m} < \delta
\]
for all $x \in F$. Consequently, $b_k \in W_{\mathbf{e}}(F, \delta)$.  

Conversely, suppose Player~II has a winning strategy 
$\tau$ in $SCFT(\Cp, \mathbf{e})$. 
Construct $A_n$ from $\omega$-covers $\U_n$ as in Theorem~\ref{thm:main_game}. 
Apply $\tau$ to obtain, at each inning, $b_n = f_{F_n, U_n} \in A_n$. Since $\tau$ is a winning strategy, it follows that $\mathbf{e} \in \overline{\{ b_n : n \in \mathbb{N} \}}$. 
Now for any finite set $F \subset X$, consider the neighborhood $W_{\mathbf{e}}(F, r)$, where $r = d(g,e)$. Then there exists $k$ such that $b_k \in W_{\mathbf{e}}(F, r)$. Also, note that each $b_k$ can be expressed as $f_{F_k,U_k}$ for some $U_k\in \U_k$ and finite subset $F_k\subset U_k$. Suppose, for a contradiction, that there is some $x \in F$ with $x \notin U_k$. By the definition of $b_k$, this gives $b_k(x) = g$, and hence $d(b_k(x), e) = r$. This contradicts the fact that $b_k \in W_{\mathbf{e}}(F, r)$. 
Thus $F \subset U_k$, and so $\{U_n : n \in \N\}$ is an $\omega$-cover. This completes the proof.
\end{proof}
\section{Corollaries}
In this section we derive consequences of Theorems~\ref{thm:main_game} 
and~\ref{cor:rothberger}. Since the $\Omega$-Menger and $\Omega$-Rothberger games on $X$ are independent of $G$, the corresponding game-theoretic tightness properties of $C_p(X,G)$ depend only on the topology of $X$, and are thus independent of the choice of the target group.
\begin{corollary}
Let $X$ be an infinite Tychonoff space and let $G_1, G_2$ be non-trivial metrizable arc-connected topological groups. Then
\begin{enumerate}
    \item[\textup{(a)}] $\textup{II} \uparrow CFT(C_p(X, G_1), \mathbf{e}_1)$ if and only if $\textup{II} \uparrow CFT(C_p(X, G_2), \mathbf{e}_2)$.
    \item[\textup{(b)}] $\textup{II} \uparrow CFT(C_p(X, G_1 \times G_2), (\mathbf{e}_1, \mathbf{e}_2))$ if and only if $\textup{II} \uparrow CFT(C_p(X, G_1), \mathbf{e}_1)$.
\end{enumerate}
The same conclusions hold for the  countable strong fan tightness 
game in place of the countable fan tightness game.
\end{corollary}
\begin{proof}
   The proof follows from  Theorems~\ref{thm:main_game} and~\ref{cor:rothberger}.  
\end{proof}
We now introduce the following notion of equivalence for group-valued function spaces.
\begin{definition}
    Two spaces $X$ and $Y$ are called $t_G$-equivalent if 
$C_p(X, G)$ and $C_p(Y, G)$ are homeomorphic as 
topological spaces. When $G = \R$, this coincides with 
the classical $t$-equivalence. Since every topological 
group isomorphism is in particular a homeomorphism, 
$G$-equivalence implies $t_G$-equivalence.
\end{definition}

The game-theoretic framework yields a preservation 
result under this weaker equivalence that is stronger 
than what is currently known at the selection-principle 
level. Sakai~\cite{Sakai} proved that the $\Omega$-Menger property is 
preserved under $l$-equivalence under an additional condition, and Osipov~\cite[Question~1]{osi4} 
asked whether the Classical Menger property is 
preserved under $t$-equivalence in general. While Osipov's question remains open at the 
selection-principle level, the following corollary shows 
that the strictly stronger property of Player~II having 
a winning strategy in the $\Omega$-Menger game is 
preserved under $t_G$-equivalence of Tychonoff spaces. 
\begin{corollary}
Let $G$ be a non-trivial metrizable arc-connected topological group, and let $X, Y$ be $t_G$-equivalent infinite Tychonoff spaces. 
Then
\begin{enumerate}
    \item[\textup{(a)}] $\textup{II} \uparrow \gfin$ if and only if $\textup{II} \uparrow \gfinY$ .
    \item[\textup{(b)}] $\textup{II} \uparrow \gone$ if and only if $\textup{II} \uparrow \goneY$ .
\end{enumerate}
\end{corollary}
\begin{proof}
Let $\varphi\colon C_p(X, G) \to C_p(Y, G)$ be a homeomorphism.
Since $C_p(X, G)$ is a topological group, composing $\varphi$ with left-translation by $\varphi(\mathbf{e}_X)^{-1}$ yields a homeomorphism sending $\mathbf{e}_X$ to $\mathbf{e}_Y$.
Hence we may assume $\varphi(\mathbf{e}_X) = \mathbf{e}_Y$.

We first show that any winning strategy for Player~II in \( CFT(C_p(X,G), \mathbf{e}_X) \) yields a winning strategy for Player~II in \( CFT(C_p(Y,G), \mathbf{e}_Y) \).

Let $\sigma$ be a winning strategy for Player~II in $CFT(C_p(X,G), \mathbf{e}_X)$.
We define a strategy $\sigma'$ for Player~II in $CFT(C_p(Y,G), \mathbf{e}_Y)$ as follows.

Suppose that in inning $n$, Player~I plays a set $A_n \subset C_p(Y,G)$ with $\mathbf{e}_Y \in \overline{A_n} \setminus A_n$ in the game on $C_p(Y,G)$. We then define
\[
A_n' = \varphi^{-1}(A_n).
\]
Since $\varphi$ is a homeomorphism with $\varphi(\mathbf{e}_X) = \mathbf{e}_Y$ and $\mathbf{e}_Y\in \overline{A_n} \setminus A_n$, we have $\mathbf{e}_X\in \overline{A_n'} \setminus A_n'$. Since $\sigma$ is a winning strategy in $CFT(C_p(X,G),\mathbf{e}_X)$, $B_n':=\sigma(A_1', \ldots, A_n')$ is a finite subset of $A_n'$ and $\mathbf{e}_X\in  \overline{\bigcup_n B_n'}$.

Now, in inning $n$, we define Player II’s response in $CFT(C_p(Y,G), \mathbf{e}_Y)$ by
\[
\sigma'(A_1,\ldots,A_n) := \varphi(B_n').
\]
Since $\varphi$ is a homeomorphism, $\varphi(B_n')$ is finite and \[
\mathbf{e}_Y = \varphi(\mathbf{e}_X) \in \varphi\!\left(\overline{\textstyle\bigcup_n B_n'}\right) = \overline{\varphi\!\left(\textstyle\bigcup_n B_n'\right)}. 
\] 
Therefore $\sigma'$ is a winning strategy for Player~II in $CFT(C_p(Y,G), \mathbf{e}_Y)$.

The reverse implication follows by applying the same argument to $\varphi^{-1}$.
Now, by Theorem~\ref{thm:main_game}, the following statements are equivalent:
\begin{align*}
\textup{II} \uparrow \gfin 
&\;\Longleftrightarrow\; \textup{II} \uparrow CFT(C_p(X,G), \mathbf{e}_X) \\
&\;\Longleftrightarrow\; \textup{II} \uparrow CFT(C_p(Y,G), \mathbf{e}_Y) \\
&\;\Longleftrightarrow\; \textup{II} \uparrow \gfinY .
\end{align*}
This completes the proof of {(a)}.

{(b)} follows identically using Theorem~\ref{cor:rothberger}.
\end{proof}

\section{Concluding Remarks}
We note that in Theorems~\ref{thm:main_game} 
and~\ref{cor:rothberger}, the forward implication 
requires only the Tychonoff property of $X$ and 
metrizability of $G$, while the converse additionally 
requires arc-connectedness; for non-arc-connected groups, 
the appropriate substitute is $G^*$-regularity of $X$ 
(see~\cite[Definition~2.2]{SS}).
The strategy translations in both directions depend only 
on the current inning: in each direction, the translated 
response at inning $n$ is constructed solely from 
Player~I's current move and the round number $n$, with 
no reference to previous innings. (The augmented set 
$F^*$ in the verification of the winning condition uses 
the history of past responses, but the strategy itself, 
that is, the rule determining Player~II's response, does not.)
This extends 
Clontz~\cite[Theorems~22 and~43]{Clontz} from $G = \R$ 
to arbitrary metrizable arc-connected groups, and yields 
the following.

\begin{proposition}
Let $X$ be an infinite Tychonoff space and $G$ a non-trivial metrizable arc-connected topological group. 
Then 
\begin{itemize}
    \item[(i)] Player~II has a winning Markov strategy in $\gfin$ if and only if Player~II has a winning Markov strategy 
in $CFT(\Cp, \mathbf{e})$.
\item[(ii)] Player~II has a winning Markov strategy in $\gone$  if and only if Player~II has a winning Markov strategy 
in $SCFT(\Cp, \mathbf{e})$. 
\end{itemize}
\end{proposition}
Finally, the assumption that $G$ is metrizable is essential in our approach. 
It remains open whether Theorem~\ref{thm:main_game} holds for non-metrizable groups; moreover, the corresponding selection-principle equivalence due to Ko\v{c}inac~\cite{Koc2} also requires this assumption.
\begin{acknow}
The first author is grateful for financial support in the form of Prime Minister’s Research Fellowship, Government of India (PMRF/2502403).
\end{acknow}
\bibliographystyle{plain}
\bibliography{reference}

\end{document}